\documentclass[a4paper,11pt]{article}

\usepackage{amsmath,amssymb,amsthm,mathrsfs}
\usepackage[utf8]{inputenc}
\usepackage{a4}
\usepackage[english]{babel}
\usepackage{enumerate}
\usepackage{graphicx}
\usepackage{color}
\usepackage[hidelinks,
  bookmarks=true,
  bookmarksopen=true,
  bookmarksnumbered=true,
  pdfauthor={Giuliano Basso},
  pdftitle={Almost minimal orthogonal projections},
  pdfsubject={Paper},
  pdfstartview=FitH,
  colorlinks=false,
  pdfencoding=unicode
]{hyperref}

\numberwithin{equation}{section}

\newtheorem{theorem}{Theorem}[section]
\newtheorem{corollary}[theorem]{Corollary}
\newtheorem{lemma}[theorem]{Lemma}
\newtheorem{proposition}[theorem]{Proposition}

\newtheorem{question}[theorem]{Question}

\theoremstyle{definition}
\newtheorem{definition}{Definition}[section]

\theoremstyle{remark}

\DeclareMathOperator{\R}{\mathbb{R}}

\DeclareMathOperator{\Tr}{Tr}

\DeclareMathOperator{\Sgn}{Sgn}

\newcommand*\norm[1]{\lVert#1\rVert}
\newcommand\abs[1]{\left\lvert#1\right\rvert}

\newcommand{\eval}[2][\right]{\relax
  \ifx#1\right\relax \left.\fi#2#1\rvert}

\let\abs=\envert

\let\norm=\enVert

\title{Almost minimal orthogonal projections}

\author{Giuliano Basso}

\date{\today}

\begin{document}


\pdfbookmark[0]{Almost minimal orthogonal projections}{titleLabel}

\maketitle

\begin{abstract}
The projection constant \(\Pi(E):=\Pi(E, \ell_\infty)\) of a finite-dimensional Banach space \(E\subset \ell_\infty\) is by definition the smallest norm of a linear projection of \(\ell_\infty\) onto \(E\). Fix \(n\geq 1\) and denote by \(\Pi_n\) the maximal value of \(\Pi(\cdot)\) amongst \(n\)-dimensional real Banach spaces. We prove for every \(\varepsilon >0\) that there exist an integer \(d\geq 1\) and an \(n\)-dimensional subspace \(E\subset \ell_1^d\) such that \(\Pi_n \leq \Pi(E, \ell_1^d) +2 \varepsilon\) and the orthogonal projection \(P\colon \ell_1^d\to E\) is almost minimal in the sense that \( \norm{P} \leq \Pi(E, \ell_1^d)+\varepsilon\). As a consequence of our main result, we obtain a formula relating \(\Pi_n\)  to smallest absolute value row-sums of orthogonal projection matrices of rank \(n\). 

\end{abstract}

\section{Introduction}

\subsection{Overview} Let \(E\subset F\) denote a finite-dimensional subspace of a real Banach space \(F\). The \textit{projection constant of \(E\) relative to \(F\)}, denoted by \(\Pi(E,F)\), is by definition the minimal norm of a linear projection of \(F\) onto \(E\), and the real number \(\Pi(E):=\Pi(E, \ell_\infty)\) is called \textit{(absolute) projection constant} of \(E\).\footnote{In the literature the commonly used symbols to denote these quantities are \(\lambda(E,F)\) and \(\lambda(E)\). See 'Comments on notation' at the end of the introduction for a short justification for why we deviate from this tradition.} These concepts are essential tools of Banach space theory. The exact values of the projection constants of certain classical spaces have been computed in \cite{MR114110, MR190708, MR412775, MR1695788}, general bounds may be found in \cite{MR3443, MR291770, MR793850, MR968885, MR2725896}, and further results are obtained in \cite{MR146649, MR0169062}. Moreover, projection constants have been used in approximation theory  \cite{MR0265842, MR358188, MR0463776, MR1782730}, and metric geometry \cite{MR167821, MR2277209, MR2288741}.

The starting point of the present article is the following observation: For many classical polyhedral spaces \(E\subset \ell_\infty^d\), the \textit{orthogonal projection} \(P\colon \ell_\infty^d \to E\) given by \(P(x)=\langle x, b_1 \rangle_{\R^d} \,b_1+\dotsm+\langle x, b_n \rangle_{\R^d} \,b_n\), where \(b_1, \ldots, b_n\) is any basis of \(E\) that is orthonormal with respect to the standard inner product \(\langle \cdot, \cdot \rangle_{\R^d}\), is a \textit{minimal projection}, that is, \(\norm{P}=\Pi(E, \ell_\infty^d)=\Pi(E)\). For instance, Chalmers \cite{MR862235} established that the orthogonal projection is minimal if the unit ball of \(E\) is a regular polytope or if the unit ball of \(E\) is the \(A_n\) root polytope. Furthermore,  a classical result of Lozinski states that the orthogonal projection (Fourier projection) \(P\colon C\to P_n\) from the space \(C\) of all continuous and \(2\pi\)-periodic functions \(f\colon [-\pi, \pi] \to \R\) equipped with the sup-norm onto the space \(P_n\) of all trigonometric polynomials of degree \(\leq n\) is minimal, cf. \cite{MR0026699, MR256044}. 

Besides these special results, an important class of examples may be constructed with the help of equiangular lines. Indeed, if \(E\) is a finite-dimensional Banach space such that the lines through opposite vertices of the unit ball of the dual space of \(E\) form a system of equiangular lines of cardinality \(d\), then \(E\subset \ell_\infty^d\) and the orthogonal projection \(\ell_\infty^d\to E\) is minimal. This result is due to K\"onig, Lewis, and Lin \cite{MR722257}. These spaces (if they exist) maximize 
\[\Pi(n, d):=\max\bigl\{ \Pi(E) : \dim(E)=n \textrm{ and } E\subset \ell_\infty^d \bigr\},\]
the maximal relative projection constant of \(n\)-dimensional Banach spaces with unit ball having at most \(2d\) faces. Thus, for example, the \(2\)-dimensional real Banach space \(E_{hex}\) with the regular hexagon as unit ball has projection constant equal to \(\Pi(2,3)\). 

In the 1960s, Gr\"unbaum \cite{MR114110} asked if \(E_{hex}\) has a maximal projection constant amongst \(2\)-dimensional real Banach spaces. In \cite{MR2725896}, Chalmers and Lewicki  resolved Gr\"unbaums conjecture by showing that \(\Pi(E_{hex})=\Pi_2=\frac{4}{3}\), where 
\[\Pi_n:=\max\bigl\{ \Pi(E) : \dim(E)=n \bigr\}\]
 denotes the \textit{maximal projection constant of order \(n\)}. It is possible to show that \(E_{hex}\subset \ell_1^3\) and \(\norm{P}=\Pi(E_{hex}, \ell_1^3)=\Pi_2\), where \(P\colon \ell_1^3\to E_{hex}\) is the orthogonal projection of \(\ell_1^3\) onto \(E_{hex}\). In this paper, we are interested if such a space also exists for \(n>2\), see Question \ref{qe:conjecture}. We proceed with a brief summary of known results for \(n=3\).

The computation of \(\Pi_n\) is an exceedingly difficult task. Only the values \(\Pi_1=1\) by the Hahn-Banach theorem and \(\Pi_2=\frac{4}{3}\) due to Chalmers and Lewicki are known. It is expected that \(\Pi_3=\Pi(E_{dod})=\phi\), where \(\phi\) denotes the golden ratio and \(E_{dod}\) is the \(3\)-dimensional Banach space whose unit ball is a dodecahedron. In \cite{MR2052228},  K\"onig and Tomczak-Jaegermann established that the Banach space \(Y_3\subset \ell_1^6\) with unit ball a icosidodecahedron, a polyhedron with twenty regular triangles and twelve regular pentagons as its faces, has the property that \(\Pi(Y_3, \ell_1^6)=\Pi(E_{dod})=\phi\). It is not hard to check that the orthogonal projection \(P\colon \ell_1^6 \to Y_3\) is minimal, so \(\norm{P}=\Pi(Y_3, \ell_1^6)=\phi\). Thus, if \(\Pi_3=\phi\), then \(\norm{P}=\Pi(Y_3, \ell_1^6)=\Pi_3\). In view of these examples, the following question arises naturally:

\begin{question}\label{qe:conjecture}
Fix an integer \(n\geq 1\). Does there exist an \(n\)-dimensional Banach space \(E\subset \ell_1^d\) for some integer \(d\geq n\) such that the orthogonal projection \(P\colon \ell_1^d \to E\) is minimal and \(\norm{P}=\Pi(E, \ell_1^d)=\Pi_n\)? 
\end{question}

We suspect that Question \ref{qe:conjecture} has a positive answer for \(n=3\). However, for \(n\geq 4\) the general picture is unclear.  It appears reasonable to expect that Question \ref{qe:conjecture} has only a positive answer for certain integers \(n\).

\subsection{Main result} Our main result implies that Question \ref{qe:conjecture} has an `almost' positive answer for every integer \(n\geq 1\). 

\begin{theorem}\label{thm:main1}
Let \(n\geq 1\) be an integer. For every \(\varepsilon >0\) there exist an integer \(d\geq 1\) and an \(n\)-dimensional subspace \(E\subset \ell_1^d\) such that
\[ \Pi_n -\varepsilon\leq \norm{P} \leq \Pi(E, \ell_1^d)+\varepsilon,\]
where \(P\colon \ell_1^d \to E\) denotes the orthogonal projection of \(\ell_1^d\) onto \(E\). 
\end{theorem}

The orthogonal projection \(P\) from the theorem above is `almost minimal' in the sense that \(0\leq \norm{P}-\Pi(E, \ell_1^d)\leq \varepsilon\). Moreover, our proof of Theorem \ref{thm:main1} shows
\[\max_{i=1, \ldots, d} \, \sum_{j=1}^d \abs{p_{ij}}\leq \min_{i=1, \ldots, d} \, \sum_{j=1}^d \abs{p_{ij}} +\varepsilon;\]
hence, the matrix \(\tfrac{1}{\norm{P}}\abs{P}\) is nearly doubly-stochastic. We use \(\abs{P}\) as shorthand notation for the matrix that is obtained from \(P\) by taking the absolute value of every entry of \(P\). The proof of Theorem \ref{thm:main1} is given in Section \ref{sec:AOP}. It heavily relies on tools from matrix analysis such as the Perron-Frobenius Theorem and a classical result of Fan \cite[Theorem 1]{MR34519}. A key step will be to make the approximation of \(\Pi_n\) from \cite{MR4001080} quantitative by the use of Dirichlet's Theorem on simultaneous approximation \cite[Theorem 1A p. 27]{MR568710}. 

By looking at the proof of Theorem \ref{thm:main1}, it is readily verified that Theorem \ref{thm:main1} remains valid if \(\ell_1^d\) is replaced by the overspace \(\ell_\infty^d\). Hence, the following corollary is a direct consequence of the modified version of Theorem \ref{thm:main1}.

\begin{corollary}\label{cor:main1}
For every integer \(n\geq 1\) there exists an \(n\)-dimensional Banach space \(E\) such that \(\Pi(E)=\Pi_n\) and \(\Pi(E)\) is the infimum of \(\norm{T_2}\, \norm{T_1}\) taken over all factorizations \(\textrm{\normalfont id}_E=T_2 T_1\), where \(T_1\colon E \to \ell_\infty^d \) and \(T_2\colon \ell_\infty^d \to E\) are linear maps such that \(T_1 T_2\) is the orthogonal projection of \(\ell_\infty^d\) onto \(T_1(E)\) and \(d\) is any integer. 
\end{corollary}

It is well-known that for every finite-dimensional Banach space \(E\) the projection constant \(\Pi(E)\) is 
the infimum of \(\norm{T_2}\, \norm{T_1}\) taken over all factorizations \(\textrm{\normalfont id}_E=T_2 T_1\), where \(T_1\colon E \to \ell_\infty^d \) and \(T_2\colon \ell_\infty^d \to E\) are linear maps and \(d\) is any integer, cf. \cite[Paragraph 4.11]{MR902804}. We hope that Corollary \ref{cor:main1} will be useful for computing exact values of \(\Pi_n\) for \(n\geq 3\). 

\subsection{A formula for \(\Pi_n\)} The subsequent characterization of the maximal projection constant \(\Pi_n\) is a key component of the proof of Theorem \ref{thm:main1}. 

\begin{theorem}\label{thm:maximizer}
Fix an integer \(n\geq 1\). We define the following:
\begin{enumerate}
\item Let \(\mathcal{S}_d\) denote the set of all \(d\times d\) symmetric \((-1,1)\)-matrices that have only ones on the diagonal and set \(\mathcal{S}:=\bigcup_{d\geq 1} \mathcal{S}_d\). We write
\[A:=\sup \Bigl\{ \frac{1}{d}\sum_{i=1}^n \lambda_i(S)\, :\,\lambda_1(S)\geq \dots \geq \lambda_d(S) \textrm{ eigenvalues of } S \in \mathcal{S}\Bigr\}. \]
\item  Let \(\mathcal{P}_{n,d}\) denote the set of all \(d\times d\) orthogonal projection matrices of rank \(n\). We abbreviate \(\mathcal{P}_n:=\bigcup_{d\geq n} \mathcal{P}_{n,d}\) and set
\[B:=\sup \bigl\{ \rho(\abs{P}) : P\in \mathcal{P}_n \textrm{ and } \abs{P} \textrm{ is positive} \bigr\}. \]
\end{enumerate}
Then \(A=B=\Pi_n.\) Moreover, there exists a matrix \(P\in \mathcal{P}_{n}\) such that \(\abs{P}\) is a positive matrix and \(\rho(\abs{P})=B\).  
\end{theorem}

We use \(\rho(\abs{P})\) to denote the spectral radius of the matrix \(\abs{P}\).  
The idea to express \(\Pi_n\) via the maximal sum of \(n\) eigenvalues of matrices taken from a certain class is due to Chalmers and Lewicki \cite{MR2527029}. A characterization of \(\Pi_n\) in terms of eigenvalues of certain two-graphs has been obtained in \cite{MR4001080}. The proof of Theorem \ref{thm:maximizer} is given in Section \ref{seq:four}. We proceed with two straightforward consequences of Theorem \ref{thm:maximizer}. 

If \(P\in \mathcal{P}_{n}\) is a \(d\times d\) matrix such that \(\abs{P}\) is positive, then by the Perron-Frobenius theorem, see for example \cite[Theorem 8.1.26]{MR2978290}, and Theorem \ref{thm:maximizer},
\begin{equation*}\label{eq:Perron}
r(\abs{P}):=\min_{i=1, \ldots, d} \, \sum_{j=1}^d \abs{p_{ij}} \leq \rho(\abs{P}) \leq \Pi_n.
\end{equation*}
As it turns out, \(\Pi_n\) is equal to the supremum taken over all lower bounds \(r(\abs{P})\) of \(\Pi_n\). 

\begin{corollary}\label{cor:lower}
Let \(n\geq 1\) be an integer. Then
\begin{equation*}\label{eq:lowerPerron}
\Pi_n=\sup\big\{ r(\abs{P}) : P\in \mathcal{P}_n \textrm{ and } \abs{P} \textrm{ is positive} \big\}.
\end{equation*}
\end{corollary}

Corollary \ref{cor:lower} is a direct consequence of our main result and its proof may be found at the end of Section \ref{sec:AOP}. Matrices \(P\in \mathcal{P}_n\) that attain the supremum in Corollary \ref{cor:lower} are of particular interest. In fact, it is not hard to check that if such a matrix \(P\) exists, then \(\norm{P}=\Pi(E, \ell_1^d)=\Pi_n\), where \(E:=P(\ell_1^d)\subset \ell_1^d\). 
In a similar spirit, if the supremum \(A\) from Theorem \ref{thm:maximizer} is attained, then Question \ref{qe:conjecture} also has a positive answer. 
This is the content of the proposition below:

\begin{proposition}\label{prop:SupA}
Fix integers \(d\geq n\geq 1\). The following statements are equivalent:
\begin{enumerate}
\item The supremum \(A\) defined in Theorem \ref{thm:maximizer} is attained by some \(d\times d\) matrix \(S\in \mathcal{S}\). 
\item There exists an \(n\)-dimensional subspace \(E\subset \ell_\infty^d\) such that the orthogonal projection \(P\colon \ell_\infty^d \to E\) is minimal and \(\norm{P}=\Pi(E)=\Pi_n\).
\item There exists an \(n\)-dimensional subspace \(E\subset \ell_1^d\) such that the orthogonal projection \(P\colon \ell_1^d \to E\) is minimal and \(\norm{P}=\Pi(E, \ell_1^d)=\Pi_n\).
\end{enumerate}
\end{proposition}

The implication \((1)\Longrightarrow (2)\) already appears in \cite[Proposition 13]{MR3566474}. 
If the Banach space \(E\) of Corollary \ref{cor:main1} is polyhedral, that is \(E\subset \ell_\infty^d\) for some integer \(d\), then it is plausible that \(\norm{P}=\Pi(E)=\Pi_n\), where \(P\colon \ell_\infty^d\to E\) denotes the orthogonal projection of \(\ell_\infty^d\) onto \(E\). Hence, in this case Proposition \ref{prop:SupA} would imply that Question \ref{qe:conjecture} has a positive answer. We do not know if \(\Pi(\cdot)\) admits non-polyhedral maximizers amongst \(n\)-dimensional Banach spaces. 

\subsection{A comment on notation} In the literature, the symbols \(\lambda(E,F)\) for the projection constant of \(E\) relative to \(F\) and \(\lambda(E)\) for the absolute projection constant of \(E\) are commonly used. This naturally leads to the shorthand notation \(\lambda_n\) for the maximal projection constant of order \(n\), cf. \cite{MR2725896}. In view of the formula \(A=\lambda_n\) from Theorem \ref{thm:maximizer} this notation could cause unnecessary confusion. This is why we switched to the greek letter \(\Pi\) to denote projection constants. This notation also appears in the monograph \cite[p. 49]{MR2882877}.  

\section{Relative projection constants}

\subsection{A formula for \(\Pi(E,F)\)}Let \(F=(\R^d, \norm{\cdot}_F)\) denote a Banach space and let \(\mathcal{L}(F)=(\mathcal{M}_d(\R), \norm{\cdot})\) be the Banach space of bounded linear operators from \(F\) into \(F\) equipped with the operator norm. The dual space of \(\mathcal{L}(F)\) is naturally identified with \((\mathcal{M}_d(\R), \nu_1(\cdot))\) via trace-duality, that is,
\[\nu_1(S)=\sup\big\{ \Tr( S T) : T\in \mathcal{M}_d(\R) \textrm { and } \norm{ T} =1 \big\} \quad \quad (S\in \mathcal{M}_d(\R)).\]
The norm \(\nu_1(\cdot)\) is called \textit{\(1\)-nuclear norm}. In general, it is not an easy task to compute \(\nu_1(S)\). But if \(F=\ell_\infty^d\), (\(F=\ell_1^d\) respectively), then
\[\nu_1(S)=\sum_{j=1}^d \, \norm{ S e_j}_{\infty}, \quad \quad \quad (\nu_1(S)=\sum_{j=1}^d \, \norm{ S^t e_j}_{\infty} \textrm{ respectively}).\] 
We will only work with overspaces \(F\) that are of this form. The following lemma is well established, cf. \cite[Lemma 1]{MR722257}, \cite[Lemma 32.3]{MR993774} or \cite[Lemma 3.12]{MR538861}. 

\begin{lemma}\label{lem:RelProj}
Let \(F=(\R^d, \norm{\cdot}_F)\) be a Banach space and suppose \(E\subset F\) is a linear subspace. Then
\[\Pi(E,F)=\max\big\{ \Tr(AP) : A\in \mathcal{M}_d(\R),\, \nu_1(A)=1 \textrm{ and } AP=PAP\big\},\]
where \(P\) denotes the transformation matrix of the orthogonal projection \(F\to E\). 
\end{lemma}

Theorem \ref{thm:maximizer} and the lemma above yield an \(n\)-dimensional subspace \(E\subset \ell_1^d\) with \(\Pi(E)=\Pi(E, \ell_1^d)=\Pi_n\). Further results in this direction have been obtained by K\"onig and Tomczak-Jaegermann, cf. \cite{MR2052228}. 

\begin{proposition}
Let \(n\geq 1\) be an integer. There exists an integer \(d\geq 1\) and an \(n\)-dimensional subspace \(E\subset \ell_1^d\) such that \(\Pi(E)=\Pi(E, \ell_1^d)=\Pi_n\). 
\end{proposition}

\begin{proof}  
From Theorem \ref{thm:maximizer} we get a matrix \(P\in \mathcal{P}_n\) such that \(\abs{P}\) has only positive entries and \(\rho(\abs{P})=\Pi_n\). Let \(v=(v_1, \ldots, v_d)\) denote the unique positive vector with 
\(\norm{v}_2=1\) and \(\abs{P} v = \Pi_n v\) and set \(D:=\textrm{diag}(v_1^2, \ldots, v_d^2)\). Further, set \(S:=\Sgn(P)\) and let \(Q\in \mathcal{P}_n\) be a matrix such that \(\sqrt{D} S \sqrt{D}\) and \(Q\) commute.
Using a result due to Fan, cf. \cite[Theorem 1]{MR34519} or \eqref{eq:KyFan}, we estimate
\begin{equation*}
\begin{split}
\Pi_n=v^t \abs{P} v&=\Tr(\sqrt{D} \Sgn(P) \sqrt{D} P) \\
&\leq \Tr(\sqrt{D} S \sqrt{D} Q) \\
&=\Tr(D S \sqrt{D} Q \sqrt{D^{-1}}).
\end{split}
\end{equation*}
The matrix \(R:= \sqrt{D} Q \sqrt{D^{-1}}\) is a projection matrix of rank \(n\). We set \(E:=R(\ell_1^d)\subset \ell_1^d\). By construction, \(R\) and \(DS\) commute and thereby \(DS P_0=P_0 DS P_0\), where \(P_0\) denotes the transformation matrix of the orthogonal projection \(\ell_1^d\to E\). 
Thus, we may invoke Lemma \ref{lem:RelProj} to conclude that
\[\Tr(D S R) =\Tr(D S P_0)\leq \Pi(E, \ell_1^d).\]
Hence, the space \(E\subset \ell_1^d\) has the desired properties. 
\end{proof}

\subsection{Sums of \(n\)-largest eigenvalues} Let \(M\) be a real \(d\times d\) matrix and let \(\lambda_1(M), \ldots, \lambda_d(M)\in \mathbb{C}\) denote the eigenvalues of \(M\). Fix an integer \(1\leq n\leq d\). We say a subset \(\{z_1, \ldots, z_n\}\subset  \mathbb{C}\) is \textit{closed under complex conjugation} (which we shorten to cucc) if \(\big\{z_1, \ldots, z_n\big\}=\big\{\overline{z_1}, \ldots, \overline{z_n}\big\}\). We set
\[\pi_n(M):=\sup\Bigl\{ \sum_{i=1}^n \lambda_{\sigma(i)}(M) : \sigma\in S_d, \big\{\lambda_{\sigma(1)}(M), \ldots, \lambda_{\sigma(n)}(M)\big\} \textrm{ is cucc}\Bigr\}.\]
By definition, \(\sup \varnothing=-\infty\). Our interest in these special sums of eigenvalues stems from the fact that they may be used to compute the exact values of the maximal relative projection constants. 

\begin{lemma}\label{lem:mats}
Let \(d\geq n \geq 1\) be integers. Then
\[\Pi(n,d):=\max\big\{ \pi_n(SD) : S\in \widehat{\mathcal{S}}_d \textrm{ and } D\in \mathcal{D}_d  \big\}, \]
where \(\widehat{\mathcal{S}}_d\) denotes the set of all matrices \(\widehat{S}\in \mathcal{M}_d(\R)\) such that the absolute value of every entry of \(\widehat{S}\) is less than or equal to one, and \(\mathcal{D}_d\) consists of all diagonal matrices \(D\in \mathcal{M}_d(\R)\) with non-negative diagonal entries and \(\Tr(D)=1\).
\end{lemma}

\begin{proof}
Direct consequence of Lemma \ref{lem:RelProj} and the fact that every real square matrix admits a real Schur form, cf. \cite[Theorem 2.3.4. (b)]{MR2978290}. 
\end{proof}
Lemma \ref{lem:mats} still holds true if the class of matrices \(\widehat{\mathcal{S}}_d\) is restricted to a finite subset \(\mathcal{S}_d\subset \widehat{\mathcal{S}}_d\). Indeed, Chalmers and Lewicki, cf. \cite[Theorem 2.3]{MR2725896}, have established that
\begin{equation}\label{eq:ChalmersLewicki}
\Pi(n,d)=\max\big\{ \pi_n(\sqrt{D} S \sqrt{D} ) : S\in \mathcal{S}_d, D\in \mathcal{D}_d \, \},
\end{equation}
where \(\mathcal{S}_d\) denotes the set of all \(d\times d\) matrices \(S\in \mathcal{S}\). 

\section{Auxiliary results from matrix analysis}

In this section, we gather several results from matrix analysis that will be used repeatedly in the proofs of Theorem \ref{thm:main1} and Theorem \ref{thm:maximizer}. 
\subsection{Equality case of an inequality due to Fan}
Let \(A\) be a symmetric \(d\times d\) matrix and let \(\lambda_1(A)\geq \ldots \geq \lambda_d(A)\) be the eigenvalues of \(A\). Fix an integer \(1\leq n \leq d\).
A well-known result of Fan, cf. \cite[Theorem 1]{MR34519}, states that
\begin{equation}\label{eq:KyFan}
\pi_n(A):=\sum_{i=1}^n \lambda_i(A)=\max\big\{\Tr(AP) : P\in \mathcal{P}_{n,d} \, \big\},
\end{equation}
where \(\mathcal{P}_{n,d}\) denotes the set of all \(d\times d\) matrices \(P\in \mathcal{P}_n\). 
Clearly, if \(A v_i=\lambda_i(A) v_i\), for \(i=1, \ldots, n\), and the vectors \(v_i\) are orthonormal, then the orthogonal projection \(P\) onto the linear span of \(v_1, \ldots, v_n\) is a maximizer of the right hand side of \eqref{eq:KyFan}. The following elementary lemma shows that every maximizer is of this form.  

\begin{lemma}\label{lem:EqCase}
Let \(A\) be a symmetric  \(d\times d\) matrix and let \(1 \leq n \leq d\) be an integer.
If \(P\in \mathcal{P}_{n,d}\) satisfies \[\pi_n(A)=\Tr(AP),\]
then \(A\) and \(P\) commute. 
\end{lemma}

\begin{proof}
Let \(v_1, \ldots, v_d\in \R^d\) be orthonormal eigenvectors of \(A\) such that \(A v_i =\lambda_i v_i\), where \(\lambda_1 \geq \dotsm \geq \lambda_d\) are the eigenvalues of \(A\). Further, let \(u_1, \ldots, u_d\in \R^d\) be orthonormal eigenvectors of \(P\) such that \(Pu_i=u_i\) for all \(1 \leq i \leq n\). We get
\begin{equation}\label{eq:equality1}
\sum_{i=1}^n \lambda_i=\Tr(AP)=\sum_{i=1}^d  \alpha_i \lambda_i,
\end{equation}
where \(\alpha_i:=\sum_{j=1}^n \abs{\langle v_i, u_j \rangle}^2\). Note that \(0\leq \alpha_i \leq 1\) and \(\alpha_1+\dotsm \alpha_d=n\). Write \(\varepsilon_i:=1-\alpha_i\), \(\varepsilon:=\varepsilon_1+\dotsm+\varepsilon_n\) and \(\alpha:=\alpha_{n+1}+\dotsm+\alpha_d\). Clearly, \(\alpha=\varepsilon\).  We estimate
\[\sum_{i=1}^d \alpha_i \lambda_i \leq (\varepsilon+\alpha_n) \lambda_n+\sum_{i=1}^{n-1} \alpha_i \lambda_i \leq \sum_{i=1}^n \lambda_i.\]
Now, from \eqref{eq:equality1} we may conclude that the above inequalities are equalities, so
\begin{equation*}
\begin{cases}
\alpha_i=1 \textrm{ or } \lambda_i=\lambda_n & \textrm{ for all } 1 \leq i \leq n, \\
\alpha_i=0 \textrm{ or } \lambda_i=\lambda_n & \textrm{ for all } n \leq i \leq d.
\end{cases}
\end{equation*} 
In particular, those vectors \(v_i\) that do not belong to the eigenspace of \(A\) associated to \(\lambda_n\)
are orthonormal eigenvectors of \(P\). This implies that \(A\) and \(P\) are simultaneously diagonalizable and thereby commute, as desired.
\end{proof}

\subsection{Sign patterns of maximizers of \(\Pi(n,d)\)}  We shall need the following lemma
which shows that for a maximizer of \eqref{eq:ChalmersLewicki} there exist a matrix in \(\mathcal{P}_{n,d}\) such that both matrices have the same sign-pattern.  

\begin{lemma}\label{lem:ComMax}
Let \(d> n\geq 1\) be integers and suppose \(D\in \mathcal{D}_d\) is a diagonal matrix with positive diagonal entries. Then
\[\max\big\{ \pi_n( \sqrt{D} \, \widehat{S} \sqrt{D}) : \widehat{S}\in \widehat{\mathcal{S}}_d \big\}=\max\big\{ \pi_n( \sqrt{D} S \sqrt{D}) : S\in \mathcal{S}_d \big\}.\] 
Moreover, if \(S\in \widehat{\mathcal{S}}_d\) is a symmetric matrix such that \(\pi_n( \sqrt{D} S \sqrt{D})\) is maximal amongst  \(\pi_n( \sqrt{D} S^\prime \sqrt{D})\), for \(S^\prime \in \widehat{\mathcal{S}}_d \),
then \(S\in \mathcal{S}_d\) and for every matrix \(P\in \mathcal{P}_{n}\) with \(\pi_n( \sqrt{D} S \sqrt{D})=\Tr(\sqrt{D} S \sqrt{D}\, P)\) it holds \(\Sgn(P)=S\). In particular, \(\abs{P}\) is a positive matrix. 
\end{lemma}

We use the symbol \(\Sgn(A)\) to denote the \textit{sign pattern matrix} of a matrix \(A\), that is, the \((i,j)\)-entry of \(\Sgn(A)\) is equal to \(-1\) if \(a_{ij}<0\), \(0\) if \(a_{ij}=0\), and \(1\) if \(a_{ij}>0\).

\begin{proof}[Proof of Lemma \ref{lem:ComMax}]
Let \(\widehat{S}\in \widehat{\mathcal{S}}_d\) be a matrix such that \(\pi_n( \sqrt{D} \, \widehat{S} \sqrt{D})\) is maximal amongst \(\pi_n( \sqrt{D} S^\prime \sqrt{D})\), for \(S^\prime \in \widehat{\mathcal{S}}_d \). Due to a result of Fan, cf. \cite[Theorem 2]{MR33981}, we estimate
\[\pi_n\bigl(\sqrt{D} \, \widehat{S} \sqrt{D}\bigr)\leq \pi_n\bigl(\sqrt{D}\tfrac{1}{2}( \widehat{S}+\widehat{S}^t)\sqrt{D}\bigr). \]
Now, we prove the moreover part of the lemma. By Lemma \ref{lem:EqCase}, \(A:=\sqrt{D} S \sqrt{D}\) and \(P\) commute. We claim that \(\abs{P}\) is positive. Suppose there exists an entry \(p_{ij}\) of \(P\) that is equal to zero. It holds
\[\pi_n\left(A\right)= \Tr(AP)=\Tr(A_0P), \]
where \(A_0:=\sqrt{D}\bigl(S-\alpha(e_i e_j^t+e_j e_i^t)\bigr)\sqrt{D}\), where \(\alpha=1\) if \(s_{ij}\geq 0\) and \(\alpha=-1\) otherwise. 
Using \eqref{eq:KyFan}, we obtain
\begin{equation}\label{eq:aux1}
\Tr(A_0P)\leq \pi_n(A_0) \leq \pi_n(A),
\end{equation}
so \(\Tr(A_0 P)=\pi_n(A_0)\). Now, Lemma \ref{lem:EqCase} tells us that \(A_0\) and \(P\) commute. This amounts to
\begin{equation}
\begin{cases}
p_{ii}=p_{jj} & \\
p_{ik}=0 & \textrm{ for all } k\neq i, \\
p_{jk}=0 & \textrm{ for all } k\neq j.
\end{cases}
\end{equation}
By repeating the argument above for each entry \(p_{ik}\), where \(k\neq i\), it follows that \(P\) is a constant multiple of the identity. Because of  \(P\in \mathcal{P}_n\) we obtain \(d=\Tr(P)=n\). Hence, we have shown that \(\abs{P}\) is positive provided that \(d > n \). Since
\[Tr(AP)=\sum_{i,j=1}^d \sqrt{d_i d_j} \,  s_{ij} \,p_{ij} \leq \Tr(\sqrt{D}\Sgn(P)\sqrt{D} P)\leq \pi_n(\sqrt{D} \Sgn(P) \sqrt{D}) \]
and \(\abs{P}\) is a positive matrix, we may use the maximality of \(\pi_n(A)\) to conclude that \(\Sgn(P)=S\). This completes the proof.  
\end{proof}

\subsection{Spectral gaps}Let \(P\in \mathcal{P}_n\) be a matrix and let \(\lambda_1 \geq  \ldots \geq \lambda_d\) be the eigenvalues of \(\abs{P}\). The following lemma gives an upper bound of \(c:=\tfrac{\lambda_2}{\lambda_1}\) provided that \(\lambda_1=\rho(\abs{P})\) is close enough to \(\Pi_n\). 

\begin{lemma}\label{lem:upperC}
Let \(P\in \mathcal{P}_n\) be an orthogonal projection matrix of rank \(n\) such that \(\abs{P}\) is a positive matrix. Write \(\lambda_1 > \lambda_2 \geq  \ldots \geq \lambda_d\) for the eigenvalues of \(\abs{P}\). If \(\lambda_1 > (\sqrt{3}-1)\sqrt{n}\), then
\[0< c < \frac{\sqrt{n}}{\lambda_1}-\frac{\lambda_1}{2 \sqrt{n}} <1, \]
where \(c:=\frac{\lambda_2}{\lambda_1}\). 
\end{lemma}

\begin{proof}
Since \(\Tr(\abs{P})=\Tr(P)\) and \(\Tr(\abs{P}^2)=\Tr(P^2)\) we get
\begin{equation}\label{eq:opti}
\begin{cases}
&\sum\limits_{i=2}^d \lambda_i=n-\lambda_1 \\
&\sum\limits_{i=2}^d \lambda_i^2=n-\lambda_1^2.
\end{cases}
\end{equation}
Note that \(\lambda_2>0\). Indeed, if \(\lambda_2\leq 0\), then by \eqref{eq:opti} we obtain \(\lambda_1\geq n;\) hence, \(\lambda_1^2\geq n^2\), which is not possible. Thus, we have established that \(c>0\).
Now, we prove the second estimate. Using \eqref{eq:opti}, we estimate
\[\lambda_2^2\leq \sum\limits_{i=2}^d \lambda_i^2=n-\lambda_1^2.\]
Thus, 
\[\lambda_2\leq \sqrt{n-\lambda_1^2} \leq \sqrt{n}\bigl(1-\frac{\lambda_1^2}{2n}\bigr),\]
since \(\lambda_1\leq \sqrt{n}\). 
This completes the proof. 
\end{proof}

\subsection{Blow-up of a matrix \(S\in \mathcal{S}\)}
For every matrix \(S\in \mathcal{S}_d\) we may obtain a graph \(G_S\) as follows: \(G_S:=(\{1, \ldots, d\}, E_S)\) with \(\{i,j\}\subset E_S\) if and only if \(s_{ij}=-1\). Conversely, given a finite simple graph \(G=(\{1, \ldots, d\}, E)\) let \(S_G\in \mathcal{S}\) be the matrix uniquely determined by \(s_{ij}=-1\) if and only if \(\{i,j\} \subset E\).
Clearly, for every \(S\in \mathcal{S}_d\) we have \(S_{G_S}=S\). Let \((p_1, \ldots, p_d)\) be a tuple of positive integers. The  \textit{\((p_1, \ldots, p_d)\)-blow-up} of a finite simple graph \(G=(\{1, \ldots, d\}, E)\) is by definition the graph obtained from \(G\) be replacing each vertex \(i\) with \(p_i\) distinct copies of \(i\) 
and a copy of \(i\) is adjacent to a copy of \(j\) if and only if \(i\) and \(j\) are adjacent in \(G\). 

\begin{definition}\label{def:BlowUp}
Let \(d\geq 1\) be an integer, let \((p_1, \ldots, p_d)\) be a tuple of positive integers and let \(S\in \mathcal{S}_d\) be a matrix. Let \(G\) denote the \textit{\((p_1, \ldots, p_d)\)-blow-up} of \(G_S\). 
The matrix \(S_{G}\) is called \textit{\((p_1, \ldots, p_d)\)-blow-up} of \(S\). 
\end{definition}

If the matrix \(S^\prime\) is a blow-up of \(S\in \mathcal{S}\), then the non-zero eigenvalues of \(S^\prime\) and \(S\) coincide, cf. \cite[Lemma 2.2]{MR4001080}. 
 
\subsection{Polyhedral maximizer of \(\Pi_n\)} The following lemma is a simple consequence of \eqref{eq:ChalmersLewicki} and the fact that \(\Pi(\cdot)\) admits a polyhedral maximizer amongst \(n\)-dimensional Banach spaces.

\begin{lemma}\label{lem:MaxPoly}
Let \(n\geq 1\) be an integer. Then there exist an integer \(d\geq n\), a matrix \(S\in \mathcal{S}_d\) and a matrix \(D\in \mathcal{D}_d\) with positive diagonal entries such that \(\pi_n(\sqrt{D} S \sqrt{D})=\Pi(n,d)=\Pi_n\). 
\end{lemma}

\begin{proof}
Let \(d \geq n\) denote the smallest integer such that \(\Pi(n,d)=\Pi_n\). 
The existence of such an integer is guaranteed by \cite[Theorem 1.4]{MR4001080}. Let \(S\in \mathcal{S}_d\) and \(D\in \mathcal{D}_d\) be matrices such that \(\pi_n(\sqrt{D} S \sqrt{D})=\Pi(n,d)\). Clearly, if \(n\geq 2\), then \(d>n\) and \(\Pi(n,d)>\Pi(n, d-1)\). Hence, every diagonal entry of \(D\) is positive. 
\end{proof}

\section{A formula for \(\Pi_n\)}\label{seq:four}

We begin this section with the proof of Theorem \ref{thm:maximizer} from the introduction. 
\begin{proof}[Proof of Theorem \ref{thm:maximizer}]
Fix \(d\geq n\) and let \(S\in \mathcal{S}_d\) satisfy \(\pi_n(S)=\max\{ \pi_n(S^\prime) : S^\prime \in \mathcal{S}_d\}\). We know from Lemma \ref{lem:ComMax} that there exists a matrix \(P\in \mathcal{P}_{n,d}\) such that \(\pi_n(S)=\Tr(SP)\) and \(\Sgn(P)=S\). Thus, \(\abs{P}\) is positive and 
\[\pi_n(S)=\Tr(SP)= j^t \abs{P} j, \]
where \(j\in \R^d\) denotes the all-ones vector. Clearly,
\[j^t \abs{P} j \leq \norm{j}_2^2 \, \rho(\abs{P}),\]
so \(A \leq B \). 

Fix \(\varepsilon >0 \) and let \(P\in \mathcal{P}_n\) be a matrix such that \(\abs{P}\) is a positive and \(B \leq \rho(\abs{P})+\varepsilon\). Suppose that \(P\) is a \(d\times d\) matrix and let \(v\in \R^d\) with \(\norm{v}_2=1\) be the unique positive vector such that \(\abs{P} v= \rho(\abs{P}) v\).
We abbreviate \(D:=\textrm{diag}( \abs{v_1}^2, \ldots, \abs{v_d}^2)\). It holds 
\[\rho(\abs{P})=v^t \abs{P} v=\Tr( \sqrt{D} \Sgn(P) \sqrt{D}\, P) \leq \pi_n( \sqrt{D} \Sgn(P) \sqrt{D}),\] 
where the inequality is due to a result of Ky Fan, cf. \cite[Theorem 1]{MR34519} or \eqref{eq:KyFan}. 
By invoking a result of Chalmers and Lewicki, cf. \cite[Theorem 2.3]{MR2725896} or \eqref{eq:ChalmersLewicki}, we obtain
\[\pi_n( \sqrt{D} \Sgn(P) \sqrt{D}) \leq \Pi_n,\] 
so \(B \leq \Pi_n\). 

The inequality \(\Pi_n \leq A\) follows readily from \cite[Theorem 1.2]{MR4001080}. Hence, we have established that \(A=B=\Pi_n\). 

We are left to show that there exists a matrix \(P\in \mathcal{P}_n\) such that \(\abs{P}\) is positive and \(\rho(\abs{P})=B\). From Lemma \ref{lem:MaxPoly} we get an integer \(d\geq n\), a matrix \(S\in \mathcal{S}_d\) and a matrix \(D\in \mathcal{D}_d\) with positive diagonal entries such that \(\pi_n(\sqrt{D} S \sqrt{D})=\Pi(n,d)=\Pi_n\). Now, Lemma \ref{lem:ComMax} tells us that there exists a matrix \(P\in \mathcal{P}_{n, d}\) with \( \pi_n(\sqrt{D} S \sqrt{D})= \Tr(\sqrt{D} S \sqrt{D}\, P)\) and \(\Sgn(P)=S\). We estimate
\[\Pi_n=\sum_{i,j}^d \sqrt{d_i d_j} \abs{p_{ij}} =\Tr(\abs{P} w w^t)\leq  \rho(\abs{P}),\]
where \(w:=(\sqrt{d_1}, \ldots, \sqrt{d_d})\). Therefore, \(\rho(\abs{P})=B\), as desired. 
\end{proof}

To conclude this section, we prove Proposition \ref{prop:SupA}.

\begin{proof}[Proof of Proposition \ref{prop:SupA}]  To begin, we show that  \((1)\Longrightarrow (3)\) and \((1)\Longrightarrow (2)\): Let \(S\in \mathcal{S}\) be a \(d\times d\) matrix with \(\lambda_1(S)+\dotsm+\lambda_n(S)=d A\). Clearly,  there exists a \(d\times d\) matrix \(P\in \mathcal{P}_n\) such that \(PS=SP\) and \(\Tr(SP)=A\). Set \(E:=P(\ell_1^d)\subset \ell_1^d\). By the use of Lemma \ref{lem:RelProj}, we estimate
\[\Pi_n=A=\Tr(\tfrac{1}{d}SP)\leq \Pi(E, \ell_1^d). \]
Furthermore, we have
\[\Pi_n=\Tr(\tfrac{1}{d}SP)\leq \frac{1}{d} \sum_{i,j=1}^d \abs{p_{ij}} \leq \rho(\abs{P})\leq B,\]
since every entry of \(\abs{P}\) is positive due to the fact that \(\pi_n(\tfrac{1}{d}S)\) is maximal amongst \(\pi_n(\tfrac{1}{d}S^\prime)\), for \(S^\prime \in \mathcal{S}_d\),  and Lemma \ref{lem:ComMax}. Consequently,
\begin{equation}\label{eq:EigAllsOnes}
\Pi_n=\Pi(E, \ell_1^d)=\rho(\abs{P})=\frac{1}{d}\sum_{i,j=1}^d \abs{p_{ij}}.
\end{equation}
It is not hard to check that the all-ones vector \(j\in \R^d\) is an eigenvector of \(\abs{P}\) with eigenvalue \(\rho(\abs{P})=\Pi_n\) due to \eqref{eq:EigAllsOnes}. Thus, the right hand side of \eqref{eq:EigAllsOnes} equals \(\norm{P}\) and we get \(\norm{P}=\Pi(E, \ell_1^d)=\Pi_n\). Moreover, since \(S\) and \(P\) commute, it follows that \(E\) considered as a subspace of \(\ell_\infty^d\) satisfies \(\Tr(\tfrac{1}{d}S P)\leq \Pi(E, \ell_\infty^d)\), and thus in this case \(\norm{P}=\Pi(E)=\Pi_n\), as desired. 

Next, we show that  \((3)\Longrightarrow (1)\). Suppose that \(E\subset \ell_1^d\) is a linear subspace such that the orthogonal projection \(P\colon \ell_1^d\to E\) is minimal and \(\Pi(E, \ell_1^d)=\norm{P}=\Pi_n\). By Lemma \ref{lem:RelProj} there exists a matrix \(A_0\) such that \(A_0P=PA_0P\), \(\nu_1(A_0)=1\) and \(\Tr(A_0P)=\norm{P}\). Since \(\nu_1(A_0)=1\), the matrix \(A_0\) may be written as a product \(DS\), with \(S\in \widehat{\mathcal{S}}_d\) and \(D\in \mathcal{D}_d\). Let \(\lambda_1, \cdots ,\lambda_k\) denote the non-zero eigenvalues of \(DSP\). Clearly, \(k\leq n\) and
\[\Pi_n=\Tr(DSP)=\lambda_1+\dotsm+\lambda_k\leq \pi_k(DS) \leq \Pi_k,\]
where the last inequality is due to Lemma \ref{lem:mats}; thus, \(k=n\). In particular, the operator \(DSP\) is invertible on \(E\). Let \(I\subset \{1, \ldots, d\}\) be the set of all indices \(i\) such that \(d_i>0\). Without loss of generality, we may suppose that \(I=\{1, \ldots, m\}\) for some \(m\in \{1, \ldots, d\}\). Since \(DSP\) is invertible on \(E\) and \(DS P=PDSP\), we infer that 
\begin{equation}\label{eq:zeros}
p_{ij}=p_{ji}=0 \quad \quad \, \textrm{ for all \(1\leq i \leq d\) and \(m\leq j \leq d\).} 
\end{equation}
 We compute,
\begin{equation}\label{eq:minimality4}
\norm{P}=\Tr(DSP)=\sum_{i,j=1}^d d_i s_{ij} p_{ji}\leq \sum_{i,j=1}^m d_i \abs{p_{ji}}\leq \norm{P}. 
\end{equation}
From \eqref{eq:minimality4} we get 
\[s_{ij}=\textrm{sgn}(p_{ij}) \quad \quad \textrm{ if } p_{ij}\neq 0.\]
We define the matrix \(\widehat{S}\) as follows
\[
\widehat{s}_{ij}:=
\begin{cases}
\widehat{s}_{ij}=\textrm{sgn}(p_{ij}) & \textrm{ if } p_{ij}\neq 0 \\
1 & \textrm{ otherwise}. 
\end{cases}
\] Furthermore, let \(\widehat{S}_0\) (\(P_0\) respectively) denote the principal submatrix obtained from \(\widehat{S}\) (\(P\) respectively) be keeping its first \(m\) rows and columns. Note that \(P_0\) is an orthogonal projection matrix of rank \(n\). By virtue of \eqref{eq:zeros}, \eqref{eq:minimality4} and a result of Fan, see \eqref{eq:KyFan}, we obtain
\[\Pi_n=\Tr(DSP)=\Tr(\tfrac{1}{m} \widehat{S}_0 P_0) \leq \pi_n(\tfrac{1}{m} \widehat{S}_0),\]
so, \(\pi_n(\tfrac{1}{m} \widehat{S}_0)=A\), as was to be shown. 

We are left to establish \((2) \Longrightarrow (1)\). To this end, let \(E\subset \ell_\infty^d\) be a linear subspace such that the orthogonal projection \(P\colon \ell_\infty^d\to E\) is minimal and \(\Pi(E)=\norm{P}=\Pi_n\). By Lemma \ref{lem:RelProj} there exists a matrix \(A_0\) such that \(A_0P=PA_0P\), \(\nu_1(A_0)=1\) and \(\Tr(A_0P)=\norm{P}\). Since \(\nu_1(A_0)=1\), the matrix \(A_0\) may be written as a product \(SD\) with \(S\in \widehat{\mathcal{S}}_d\) and \(D\in \mathcal{D}_d\). As before, it is possible to show that the operator \(SDP\) has \(n\) non-zero eigenvalues. Let \(I\subset \{1, \ldots, d\}\) be the set of all indices \(i\) such that \(d_i>0\). Without loss of generality, we may suppose that \(I=\{1, \ldots, m\}\) for some \(m\in \{1, \ldots, d\}\). The matrices \(SD\) and \(\sqrt{D} S \sqrt{D}\) have the same eigenvalues. We set \(S^{sym}:=\tfrac{1}{2}\left(S+S^t\right)
\). With the help of a result due to Fan, cf. \cite[Theorem 2]{MR33981}, we estimate
\[\Pi_n \leq \pi_n(SD)=\pi_n(\sqrt{D}S \sqrt{D})\leq \pi_n(\sqrt{D}S^{sym}\sqrt{D}) \leq \Pi_n.\]
Let \(S_0\) (\(S^{sym}_0\) respectively) denote the principal submatrix obtained from \(S\) (\(S^{sym}\) respectively) be keeping its first \(m\) rows and columns.
By virtue of Lemma \ref{lem:ComMax}, the absolute value of each entry of \(S^{sym}_0\) is equal to one. 
Therefore, \(S_0=S^{sym}_0\). Set \(\Lambda:=\textrm{diag}(d_1, \ldots, d_m, 1, \ldots, 1)\). Clearly, \(v\in \R^d\) is an eigenvector of \(\sqrt{D} S \sqrt{D}\) if and only if \(\sqrt{\Lambda^{-1}} v\) is an eigenvector of \(SD\). Therefore, there exist an eigenbasis of \(SD\) that is orthonormal with respect to the scalar product \(\langle x, y \rangle_{\Lambda}= x^t \Lambda y\). Now, since \(A_0P=PA_0P\) and the operator \(A_0P\) has \(n\) non-zero eigenvalues, the column space of \(P\) is contained in the column space of \(SD\) and therefore
\[\langle p_i, e_j \rangle_{\Lambda}=0 \quad \quad \, \textrm{ for all \(1\leq i \leq d\,\) and \(\,m\leq j \leq d\),}  \]
where \(p_i\in \R^d\) is the \(i\)-th column vector of \(P\). Thus, we obtain \eqref{eq:zeros} and the proof we used in \((3)\Longrightarrow (1)\) with the necessary changes being made now applies. This completes the proof. 
\end{proof}

\section{Subspaces of \(\ell_1^d\) with almost minimal orthogonal projections}\label{sec:AOP}

The subsequent proposition is the key ingredient in the proof of Theorem \ref{thm:main1}.

\begin{proposition}\label{prop:approxbyD}
Let \(n \geq 1\) be an integer. For every \(\varepsilon > 0\) there exist an integer \(d\geq n\) and a matrix \(P\in \mathcal{P}_{n,d}\) such that \(P\) and \(\Sgn(P)\) commute, the matrix \(\abs{P}\) is positive, \(\Pi_n\leq \rho(\abs{P})+\varepsilon\), and
\[j^t \abs{P} j \leq d\,\rho(\abs{P})\leq j^t \abs{P} j+\varepsilon,\]
where \(j\in\R^d\) denotes the all-ones vector. 
\end{proposition}

\begin{proof}
By Theorem \ref{thm:maximizer}, there exists a matrix  \(P_0\in \mathcal{P}_n\) such that \(\abs{P_0}\) has only positive entries and \(\rho(\abs{P_0})=\Pi_n\).  
Suppose that \(P_0\) is an \(m\times m\) matrix and let \(v\in \R^{m}\) with \(\norm{v}_2=1\) denote the unique positive vector such that \(\abs{P_0} v= \rho(\abs{P_0}) v\). The existence of such a vector is guaranteed by the Perron-Frobenius theorem. We write \(v=(\sqrt{d_1}, \ldots, \sqrt{d_{m}})\), where \(d_i>0\) for all \(1 \leq i \leq m\). Note that \(d_1+\dotsm d_m=1\) and \(\varepsilon_0:=\min\{ d_i : 1\leq i \leq m\} >0\). Fix an integer 
\begin{equation}\label{eq:kLargeenough}
k > \frac{4(m-1)\Pi_n}{\varepsilon \varepsilon_0}.
\end{equation}
By Dirichlet's Theorem on simultaneous approximation \cite[Theorem 1A p. 27]{MR568710}, there exist integers \(p_1, \ldots, p_{m-1}, d\) such that
\[1\leq d < k^{m-1} \quad \textrm{ and } \quad d\, \lvert d_i-\frac{p_i}{d}\rvert \leq \frac{1}{k} \quad \quad  (1\leq i \leq m-1).\] 
We set \(q_i:=\frac{p_i}{d}\) for all \(1 \leq i \leq m-1\) and \(q_m:=\frac{p_d}{d}:=1-(q_1+\dotsm+q_{m-1})\).
By construction, \(q_1+\dotsm+q_m=1\) and 
\begin{equation}\label{eq:fund}
d \sum_{i=1}^m\abs{d_i-q_i} \leq 2(m-1) \frac{1}{k}. 
\end{equation}
It follows from \eqref{eq:kLargeenough} and \eqref{eq:fund} that \(q_i>0\) for all integers \(1\leq i \leq m\). Let the \(d\times d\) matrix \(S\in \mathcal{S}\) denote the \((p_1, \ldots, p_m)\)-blow up of the \(m\times m\) matrix \(\textrm{Sgn}(P_0)\), see Definition \ref{def:BlowUp}, and suppose \(S_{\star}\in \mathcal{S}_d\) is a matrix such that \(\pi_n(S_\star)=\max\big\{ \pi_n(S^\prime) : S^\prime\in \mathcal{S}_d \, \big\}\).

Now, we are ready to define the matrix \(P\): By Lemmas \ref{lem:EqCase} and \ref{lem:ComMax}, there exists a matrix \(P\in \mathcal{P}_n\) such that \(\pi_n(S_\star)=\Tr(S_\star P)\), \(\Sgn(P)=S_\star\), and the matrices \(P\) and \(\Sgn(P)\) commute.
 We obtain
\begin{equation}\label{eq:UppEsT}
j\abs{P} j^t=\Tr(\Sgn(P) P)=\pi_n(S_\star)\geq \pi_n(S)=d\,\pi_n(\textrm{Sgn}(P_0) \Lambda ),
\end{equation}
where \(\Lambda:=\textrm{diag}(q_1, \ldots, q_m)\) and for the last equality we have used \cite[Lemma 2.2]{MR4001080}. Note that
\begin{equation}\label{eq:LowEsT}
d\,\pi_n(\textrm{Sgn}(P_0) \Lambda )=d\,\pi_n(\sqrt{\Lambda}\, \textrm{Sgn}(P_0) \sqrt{\Lambda}) \geq d\, \sqrt{q}^t \abs{P_0} \sqrt{q},
\end{equation}
where \(\sqrt{q}:=(\sqrt{q_1}, \ldots, \sqrt{q_m})\). The equality is a consequence of the identity \[\sqrt{\Lambda}\, \textrm{Sgn}(P_0) \Lambda \, \sqrt{\Lambda}^{-1}=\sqrt{\Lambda}\, \textrm{Sgn}(P_0) \sqrt{\Lambda}\] and the inequality is due to a theorem of K. Fan, cf. \cite[Theorem 1]{MR34519} or \eqref{eq:KyFan}. 
Thus, by \eqref{eq:UppEsT} and \eqref{eq:LowEsT}
 \[j\abs{P} j^t \geq d\, \sqrt{q}^t \abs{P_0} \sqrt{q}\] and we may compute
\begin{equation}\label{eq:first}
\begin{split}
d\, \rho(\abs{P})-j\abs{P} j^t &\leq d\left(\rho(\abs{P_0})-\sqrt{q}^t \abs{P_0} \sqrt{q}\right)\leq 2  \Pi_n \, d\,\norm{ v- \sqrt{q}}_2.
\end{split}
\end{equation}
By \eqref{eq:fund} and \eqref{eq:kLargeenough}, we get
\begin{equation}\label{eq:second}
 2\Pi_n\, d \,\norm{ v- \sqrt{q}}_2 \leq \frac{2 \Pi_n}{\sqrt{2\varepsilon_0}} \,d \sum_{i=1}^m\abs{d_i-q_i} \leq \frac{4(m-1) \Pi_n}{\sqrt{2\varepsilon_0}} \frac{1}{k} \leq \varepsilon.
\end{equation}
To obtain the first inequality, we have used that the numbers \(q_i, d_i\) all lie in the interval \([\frac{\varepsilon_0}{2}, 1]\).
From \eqref{eq:first} and \eqref{eq:second} we conclude
\[d\,\rho(\abs{P})\leq j^t \abs{P} j+\varepsilon,\]
as was to be shown. 
\end{proof}

Now, we have everything at hand to prove our main result.

\begin{proof}[Proof of Theorem \ref{thm:main1}]
We set
\begin{equation}\label{eq:EtaDef}
 \eta:=\frac{1}{\sqrt{n}} \min\bigl\{1, \left(\frac{\varepsilon}{32}\right)^2\bigr\}.
\end{equation}
From Proposition \ref{prop:approxbyD}, we get an integer \(d\geq n\) and a matrix \(P\in \mathcal{P}_{n,d}\) such that  \(P\) and \(\Sgn(P)\) commute, the matrix \(\abs{P}\) is positive, \(\Pi_n \leq \rho(\abs{P})+\eta\) and 
\begin{equation}\label{eq:DscaledEstimate}
d\,\rho(\abs{P})\leq j^t \abs{P} j+\eta,
\end{equation}
where \(j\in \R^d\) denotes the all-ones vector. Let \(v_1, \ldots, v_d\in \R^d\) with \(\norm{v_i}_2=\sqrt{d}\) be an orthogonal eigenbasis of \(\abs{P}\) such that \(\abs{P} v_i=\lambda_i \,v_i\) for all \(1\leq i \leq d\), where the eigenvalues \(\lambda_i:=\lambda_i(\abs{P})\) are ordered such that \(\lambda_1\geq \ldots \geq \lambda_d\). It holds
\begin{equation}\label{eq:esti2}
 j^t\abs{P} j= d \, \sum_{i=1}^d \Bigl(\frac{\abs{\langle j, v_i\rangle}}{d}\Bigr)^2 \lambda_i.
\end{equation}
We set \(\alpha_i:= \bigl(\frac{\abs{\langle j, v_i\rangle}}{d}\bigr)^2\), for \(1 \leq i \leq d\), and 
\[\alpha:=\sum_{i=2}^d \alpha_i.\]
Note that \(\alpha\in [0,1]\). In what follows, we show \(\alpha\to 0\) for \(\eta \to 0\).  By virtue of \eqref{eq:DscaledEstimate} and \eqref{eq:esti2},  we obtain
\[\lambda_1\leq (1-\alpha)\lambda_1+\alpha \, c \, \lambda_1+\frac{\eta}{d},\]
where \(c:=\tfrac{\lambda_2}{\lambda_1}\) and by the Perron–Frobenius theorem \(\rho(\abs{P})=\lambda_1\). Hence,
\begin{equation}\label{eq:makeAlphasmall}
\alpha (1-c)\lambda_1 \leq \frac{\eta}{d}.
\end{equation}
By results of Gr\"unbaum, cf. \cite{MR114110}, and Rutovitz, cf. \cite{MR190708},  
\[\Pi_n\geq \Pi(\ell_2^n) > \sqrt{\frac{2}{\pi}}\sqrt{n}.\]
Therefore, \(\lambda_1 \geq  \sqrt{\frac{2}{\pi}}\sqrt{n}-\eta\) and from the definition of \(\eta\) we infer  \( \lambda_1 \geq \tfrac{3}{4}\sqrt{n}> (\sqrt{3}-1)\sqrt{n}\). Hence, by the use of Lemma \ref{lem:upperC}, we get
\[c<\frac{23}{24}<1\]
and by invoking \eqref{eq:makeAlphasmall}, we may deduce that 
\[\alpha \leq \frac{c_n \eta}{d}, \quad \quad \quad\textrm{ where  }\,  \, c_n:=\frac{32}{\sqrt{n}}.\]
As a consequence, 
\[1-\frac{c_n \eta}{d} \leq \Bigl(\frac{\abs{\langle j, v_1\rangle}}{d}\Bigr)^2 \leq  1\] 
 and by the parallelogram law, 
\[\norm{j-v_1}_2^2=2\left(d- \langle j, v_1 \rangle\right)\leq 2d\Bigl(1- \Bigl(\frac{\abs{\langle j, v_1\rangle}}{d}\Bigr)^2\Bigr)\leq 2 c_n \eta.\]
Thus, \(\vert v_1^{(i)}-v_1^{(j)}\vert\leq \sqrt{8 c_n \eta} ,\) where \(v_1=(v_1^{(1)}, \ldots, v_1^{(d)})\). 
For all integers \(1\leq r,s\leq d\), we estimate
\begin{equation}\label{eq:difference}
\begin{split}
\sum_{i=1}^d \abs{p_{ri}}&\leq \Pi_n v_1^{(r)} +\sqrt{2 c_n\eta }\sqrt{\sum\nolimits_{i=1}^d \abs{p_{ri}}^2} \\
&\leq \Pi_n v_1^{(s)}+\sqrt{8 c_n \eta} \, \Pi_n +\sqrt{2 c_n \eta } \sqrt{n} \\
&\leq \sum_{i=1}^d \abs{p_{si}}+\sqrt{32 c_n n} \sqrt{\eta}. 
\end{split}
\end{equation}
This implies 
\[\norm{P}-\varepsilon \leq \frac{1}{d} \sum_{i,j=1}^d \abs{p_{ij}}.\]
We set \(E:=P(\ell_1^d)\subset \ell_1^d\). Since \(\Sgn(P)\) and \(P\) commute, Lemma \ref{lem:RelProj}
and the estimate above yield \(\norm{P}-\varepsilon \leq \Pi(E, \ell_1^d).\)
Furthermore, \(\Pi_n \leq \rho(\abs{P})+\eta \leq \norm{P}+\varepsilon\), as desired.
\end{proof}

Now we are in a position to prove Corollary \ref{cor:lower}. 

\begin{proof}[Proof of Corollary \ref{cor:lower}]
As pointed out in the introduction, 
\[\sup\bigl\{ r(\abs{P}) : P\in \mathcal{P}_n \textrm{ and } \abs{P} \textrm{ is positive} \bigr\} \leq \Pi_n.\]
Fix \(\varepsilon >0\) and define \(\eta\) as in \eqref{eq:EtaDef}. The proof of Theorem \ref{thm:main1} shows that there exists a \(d\times d\) matrix \(P\in \mathcal{P}_n\) such that estimate \eqref{eq:difference} holds and \(\Pi_n \leq \rho(\abs{P})+\eta\leq \rho(\abs{P})+\varepsilon \). Write 
\[R(\abs{P}):=\max_{i=1, \ldots, d} \, \sum_{j=1}^d \abs{p_{ij}}.\] 
By the Perron-Frobenius theorem, \(\rho(\abs{P})\leq R(\abs{P})\), and due to \eqref{eq:difference}, 
\[R(\abs{P})-r(\abs{P}) \leq \varepsilon;\]
thus, \(\Pi_n \leq R(\abs{P})+\varepsilon \leq r(\abs{P})+2 \varepsilon.\) This completes the proof. 
\end{proof}

\subsection{Acknowledgements} I am indebted to the anonymous referee for valuable suggestions. Moreover, I am grateful to Anna Bot for proofreading this paper.


\bibliographystyle{plain}
\bibliography{refs}
\noindent

\bigskip\noindent
Giuliano Basso ({\tt giuliano.basso@math.ethz.ch}),\\
Department of Mathematics, ETH Zurich, 8092 Zurich, Switzerland

\end{document}